\documentclass[12pt,a4paper,twoside]{amsart}
\usepackage{amsfonts, amsthm, amsmath, amssymb}
\usepackage{hyperref}
\usepackage{mathrsfs,amsmath}
\hypersetup{colorlinks=false}

\usepackage[margin=1.25in]{geometry}

\usepackage{helvet}
\usepackage{hyperref}

\usepackage{amsthm}
\newtheorem{theorem}{Theorem}
\newtheorem{lemma}{Lemma}[section]
\newtheorem{remark}{Remark}

\newtheorem{corollary}{Corollary}

\begin{document}

\author{Aritra Ghosh, Sumit Kumar, Kummari Mallesham and Saurabh Kumar Singh}
\title{Bounds for del Pezzo surfaces of degree two}

\address{ Aritra Ghosh, Sumit Kumar, Kummari Mallesham \newline {\em Stat-Math Unit, Indian Statistical Institute, 203 B.T. Road, Kolkata 700108, India; \newline  Email:  aritrajp30@gmail.com, sumitve95@gamil.com, iitm.mallesham@gmail.com
} }

\address{ Saurabh Kumar Singh \newline {\em Indian Institute of Technology, Kanpur, India; \newline  Email: skumar.bhu12@gmail.com
} }
\maketitle

\begin{abstract}
In this article, we obtain an upper bound for the number of integral points on the del Pezzo surfaces of degree two. 
\end{abstract}

\tableofcontents

\section{\bf Introduction}
 
 Let $F(x_{1},x_{2},x_{3})$ be be an irreducible homogeneous polynomial of degree four with integer coefficients. Then the equation 
 $$y^{2} =F(x_{1},x_{2},x_{3}) $$
 defines a del Pezzo surface of degree two (see, \cite{Tim}). We are interested in the rational points on this surface. Let 
 $$N(B) = \sharp \left\lbrace (x_{1}, x_{2}, x_{3}) \in [-B, B]^{3} : F(x_{1}, x_{2}, x_{3}) = y^2 \quad \text{for some} \, y \in \mathbb{Z} \right\rbrace.$$
 One expects to show that
 $$N(B) \ll_{F, \epsilon} B^{2+\epsilon}.$$

In general, this bound is optimal. For example, if we take $F(x_{1},x_{2},x_{3})=x_{1}^4+ x_{2}^4-x_{3}^{4}$  one would get the bound $N(B) \gg B^{2}$. In \cite{broberg}, N. Broberg obtained the bound
 $$N(B) \ll_{F, \epsilon} B^{\frac{9}{4}+\epsilon}$$
 using the p-adic determinant method (introduced by Heath-Brown \cite{HB} for hypersurfaces and extended to arbitrary varieties by N. Broberg and P. Salberger in \cite{BP}).
 In \cite{munshi1}, R. Munshi obtained the bound
 $$N(B) \ll_{F} B^{\frac{9}{4}} \left(\log B\right)^{\frac{3}{4}}$$
 using different methods than that of N. Broberg. 
 
 In fact, R. Munshi \cite{munshi1} obtained bound for an equation of the following type 
 $$y^{d} = F(x_{1},\ldots,x_{n}), \quad d \geq 2$$ 
 with  $F$ being an irreducible homogeneous polynomial of degree $md$ with $m \geq 1$. To count integer solutions to above equations, he introduced the $d$-power sieve which is based on ideas of D. R. Heath-Brown \cite{heath}. Later, T. D Browning \cite{tim2} developed the polynomial sieve which extends both the square sieve of D. R. Heath-Brown and $d$-power sieve of R. Munshi.  In \cite{lilli},  D. R. Heath-Brown and L. Pierce  improved  results of R. Munshi \cite{munshi1} for $n \geq 8$.

 \begin{remark}
 	D. Bonolis and L. Pierce spotted a gap in R. Munshi's original arguments. Later following suggestions of R. Munshi (see Remark 1 of \cite{Bonolis}) D. Bonolis filled the gap in his paper \cite{Bonolis} getting a bound of the same strength.   
 \end{remark}

 As an application of the main result in \cite{per} of P. Salberger (where he uses global determinant method through which one can look at many congruences modulo many primes simultaneously), one can easily obtain the following bound
 $$N(B) \ll B^{\frac{3}{\sqrt{2}}+\epsilon}.$$ 
 Through email communication we learnt from P. Salberger that he also has an unpublished result which  improves the above exponent $3/\sqrt{2}$ to $36/17$.

 The aim of this article is to improve the exponent $36/17$ of P. Salberger  using the ideas  in \cite{munshi1} of R. Munshi.  
 
 \begin{theorem} \label{mainth}
 We have
 $$N(B) \ll_{F, \epsilon} N^{2+\frac{1}{10}+\epsilon}.$$
 \end{theorem}
Note that $y^{2} =F(x_{1},x_{2},x_{3})$ defines a variety $V$ in the weighted projective space $\mathbb{P}(2,1,1,1)$, where $y$ is given weight $2$ and each $x_i$ is given weight $1$. This variety can also be viewed as a cyclic $2$-sheeted cover of the projective plane $\mathbb{P}^2$, via the natural map
\begin{align*} 
	g:V &\rightarrow \mathbb{P}^2; \notag\\
	(y,x_1,x_2,x_3)&\mapsto (x_1,x_2,x_3).
\end{align*}
To the above map, we associate a counting function 
\begin{align*}
	N(g,B)=\sharp\left\lbrace P \in V(\mathbb{Q}): H(g(P)) \leq B \right\rbrace.
\end{align*}
Serre's conjecture, in this case, predicts that 
$$N(g,B) \ll B^{2} \ (\log B)^{\gamma},$$
for some $\gamma <1$.
As a corollary to Theorem \ref{mainth}, we also obtain the following bound
\begin{corollary}
	Let $g:V \rightarrow \mathbb{P}^2$ be a cover of degree $2$. Then we have
	$$N(g,B) \ll B^{2+\frac{1}{10}+\epsilon},$$
	where $V$  and $g$ are defined as above.
\end{corollary}
 
 We use a variant of square sieve to prove Theorem \ref{mainth}. Initially the square sieve was introduced by D. R. Heath-Brown in \cite{HB} with prime moduli. Later L. Pierce  developed a variant of square sieve  in \cite{valil} with  composite moduli. In fact we use a version of square sieve by L. Pierce with composite moduli.  These composite moduli play an important role in our argument (serve as conductor lowering after splitting out the moduli, see Lemma\ref{midlelemma} and Lemma \ref{afterpoisson}). In its proof  we will assume that $F(x_{1},x_{2},x_{3})$ is of the form 
 $$F(x_{1},x_{2},x_{3})= \mathop{\sum \sum \sum}_{\substack{0 \leq i, j, k \leq 4 \\ i+j+k=4}} \, c_{i,j,k} \, x_{1}^{i} x_{2}^{j} x_{3}^{k}, \quad c_{i,j,k} \in \mathbb{Z}$$
 with $c_{4,0,0} = c_{0,4,0}=c_{0,0,4}=0$. In the case when one of $c_{4,0,0}, \  c_{0,4,0} $ or $c_{0,0,4}$ is non zero, then one can easily get better bounds $N(B) \ll_{F,\epsilon} B^{2+\epsilon}$ by noting the fact that any integer  $\ell$ can be represented in the form $x^{2}+cy^{2}$ with $|x|,|y|\leq B$ in at most $B^{2+\epsilon}$ ways.

 \section{\bf An application of square sieve}
Let's recall that
$$N(B) = \sharp \left\lbrace (x_{1}, x_{2}, x_{3}) \in [-B, B]^{3} : F(x_{1}, x_{2}, x_{3}) = \square \right\rbrace.$$
Let $P_{1}$ and $ P_{2}$ be two large positive real parameters such that
$$P_{2} \geq C \log B \quad \, \, \text{and} \, \  \  \, 10 P_{2} \leq P_{1}$$
for some large positive constant $C$ which depends only on the form $F$. Indeed, we choose $P_{1}$ and $ P_{2}$ (see, Section \ref{choose}) as follows:

\begin{equation} \label{choicep1p2}
P_{1}=B^{\frac{3}{5}-\epsilon} \quad \, \text{and} \, \ \  \, P_{2}= P_{1}^{\frac{1}{2}+\epsilon}.
\end{equation}
 We first detect squares in $N(B)$ as follows:
\begin{align}\label{intro of jacobi}
	N(B) \ll \, \frac{1}{(P_{1}P_{2})^2} \mathop{\sum \sum \sum}_{(x_{1}, x_{2},x_{3}) \in [-B, B]^3} \Big\vert \mathop{\sum \sum}_{\substack{p_{1} \sim P_{1} \\ p_{2} \sim P_{2}}} \left(\frac{F(x_{1}, x_{2}, x_{3})}{p_{1}p_{2}}\right) \Big\vert^{2}
\end{align}
where $\left(\frac{n}{p_1p_2}\right)$ is the Jacobi symbol and $p \sim P$ means that $P \leq p  \leq 2P$. From now on we write $\chi_{m}(n) := (\frac{n}{m})$ for any non-zero integer $m$.

Let $W: \mathbb{R}^{3} \to \mathbb{R}$ be a non-negative compactly supported smooth function supported  in  $[-2, 2]^{3}$ and satisfying $W(x_{1}, x_{2}, x_{3}) =1$ whenever $(x_{1}, x_{2}, x_{3}) \in [-1,1]^3$. Moreover
$$ \frac{\partial^{j_{1}+j_{2}+j_{3}}}{\partial x_{1}^{j_{1}} \partial x_{2}^{j_{2}} \partial x_{3}^{j_{3}}} \, W\left(x_{1}, x_{2}, x_{3}\right) \ll_{j_{1}, j_{2}, j_{3}} 1.$$ 
We now smooth out the sum over $x_i$'s in \eqref{intro of jacobi} as follows:
$$N(B) \ll \, \frac{1}{(P_{1}P_{2})^2} \mathop{\sum \sum \sum}_{(x_{1}, x_{2},x_{3}) \in \mathbb{Z}^3} \, W\left(\frac{x_{1}}{B}, \frac{x_{2}}{B}, \frac{x_{3}}{B} \, \right)\Big\vert \mathop{\sum \sum}_{\substack{p_{1} \sim P_{1} \\ p_{2} \sim P_{2}}} \chi_{p_{1}p_{2}} \left(F(x_{1}, x_{2}, x_{3})\right) \Big\vert^{2}.$$
Opening the absolute value square and  letting $q=p_{1} p_{2}$,  $q^{\prime} = p_{1}^{\prime} p_{2}^{\prime}$  and $Q= P_{1}P_{2}$, we see that the right hand side of the above expression transforms into
$$\frac{1}{Q^{2}} \mathop{\sum \sum}_{q, q^{\prime} \sim^{\star} Q} \mathop{\sum \sum \sum}_{(x_{1},x_{2},x_{3})\in \mathbb{Z}^{3}} \chi_{qq^{\prime}} \left(F(x_{1},x_{2},x_{3})\right) \, W\left(\frac{x_{1}}{B}, \frac{x_{2}}{B}, \frac{x_{3}}{B} \, \right)$$
where $q \sim^{\star} Q$  mean that $ Q \leq q \leq 4 Q$. We estimate the above sum by considering two cases when $q =q^{\prime}$ and $q\neq q^{\prime}$. Indeed, we have
$$N(B) \ll \frac{B^{3}}{Q} \, + \, \mathcal{S}(Q,B),$$
where
\begin{equation} \label{mainsum}
\mathcal{S}(Q,B) = \frac{1}{Q^{2}} \mathop{\sum \sum}_{\substack{q, q^{\prime} \sim^{\star} Q \\ q \neq q^{\prime}}} \mathop{\sum \sum \sum}_{(x_{1},x_{2},x_{3})\in \mathbb{Z}^{3}} \chi_{qq^{\prime}} \left(F(x_{1},x_{2},x_{3})\right) \, W\left(\frac{x_{1}}{B}, \frac{x_{2}}{B}, \frac{x_{3}}{B} \, \right).
\end{equation}
\section{\bf An applications of the Poisson summation formula}
In this section, we will  analyze $\mathcal{S}(Q,B)$. 
We have the following lemma.
\begin{lemma}
Let $\mathcal{S}(Q,B)$ be as in \eqref{mainsum}. Then we have
\begin{align}\label{S after poisson}
	\mathcal{S}(Q,B) = \frac{B^{3}}{Q^{2}} \mathop{\sum \sum}_{\substack{q, q^{\prime} \sim^{\star} Q \\ q \neq q^{\prime}}} \frac{1}{(qq^{\prime})^3}\mathop{\sum \sum \sum}_{{\bf x}=(x_{1},x_{2},x_{3})\in \mathbb{Z}^{3}}\, \mathfrak{C} \left(qq^{\prime}, {\bf x}\right) \, \mathfrak{I}\left(qq^{\prime}, {\bf x}\right),
\end{align}
where
\begin{equation} \label{charsum1}
\mathfrak{C} \left(qq^{\prime}, {\bf x}\right)  = \mathop{\sum \sum \sum}_{\beta_1, \beta_2, \beta_3 (\bmod q q^\prime) }   \chi_{q q^\prime} \big(  F(\beta_1, \beta_2, \beta_3)\big) e \left(  \frac{x_1 \beta_1 +x_2 \beta_2 + x_3  \beta_3 }{q q^\prime}\right)
\end{equation}
and 
\begin{equation} \label{theintegral}
\mathfrak{I}\left(qq^{\prime}, {\bf x}\right) = \iiint_{\mathbb{R}^3} W\left(y_{1}, y_{2}, y_{3}\right) \, e \left( -  \frac{x_1 y_1  B + x_2 y_2 B  +  x_3 y_3  B }{q q^\prime}\right) d y_1 \ dy_2 \ dy_3,
\end{equation}
with $x_{i} \ll (Q^2/B )B^{\epsilon}$ for all $i=1,2,3$.
\end{lemma}

\begin{proof}
Reducing each $x_i$ modulo $qq^{\prime}$, i.e., changing the variables $x_i$ to  $\beta_i+x_iqq^{\prime}$, we observe that $\mathcal{S}(Q,B)$ transforms into 
\begin{align*}
	\frac{1}{Q^{2}}& \mathop{\sum \sum}_{\substack{q, q^{\prime} \sim^{\star} Q \\ q \neq q^{\prime}}}\mathop{\sum \sum \sum}_{\beta_1, \beta_2, \beta_3 (\bmod q q^\prime)} \chi_{q q^\prime} \big(  F(\beta_1, \beta_2, \beta_3)\big) \\
	  &\times \mathop{\sum \sum \sum}_{(x_{1},x_{2},x_{3})\in \mathbb{Z}^{3}}  W\left(\frac{\beta_1+x_1qq^{\prime}}{B}, \frac{\beta_2+x_2qq^{\prime}}{B}, \frac{\beta_3+x_3qq^{\prime}}{B} \, \right).
\end{align*}
Here we used the fact that $F(x_{1},x_{2},x_{3}) \equiv   F(\beta_1, \beta_2, \beta_3)\mod qq^{\prime}$. Next, we apply the Poisson summation formula to the sum over $x_i$'s. Thus the sum over $x_i$ transforms into
\begin{align*}
	\mathop{\sum \sum \sum}_{(x_{1},x_{2},x_{3})\in \mathbb{Z}^{3}}\iiint_{\mathbb{R}^3} W\left(\frac{\beta_1+y_1qq^{\prime}}{B}, \frac{\beta_2+y_2qq^{\prime}}{B}, \frac{\beta_3+y_3qq^{\prime}}{B} \right) \\
	\times e\left(- x_1y_1-x_2y_2-x_3y_3\right)d y_1  dy_2 dy_3.
\end{align*}
Lastly, making a change of variable $\frac{\beta_i+y_iqq^{\prime}}{B} \rightarrow y_i$, we get the lemma. We get the range for $x_i$'s by using integration by parts on the integral $\mathfrak{I}\left(qq^{\prime}, {\bf x}\right)$. 
\end{proof}

\section{\bf Analysis of Character sums}
In this section, we will analyze $\mathfrak{C} \left(qq^{\prime}, {\bf x}\right)$ in \eqref{charsum1}. We start by recalling some results on bounds for these kind of character sums. 
\subsection{Multiplicative character sums}
In this subsection, we collect some results on bounds for multiplicative character sums which are due to N. Katz \cite{mixedcharsum}. Statement of the following lemma is a combination of the statements of Theorem 2.1 and Theorem 2.2 in \cite{mixedcharsum}.

\begin{lemma} \label{multibound}
Let $k$ be a finite  field of characteristic $p$ with cardinality $q$. Let $f \in k[t_{1},t_{2},\ldots, t_{n}]$ be a Deligne polynomial of degree $d \geq 1$ such that $f=0$ defines a smooth hypersurface  in $\mathbb{A}_{k}^{n}$. Let $\chi $ be a non-trivial multiplicative character on $k$. Then we have
$$\Big \vert \sum_{{\bf t} \in k^{n}} \chi \left(f({\bf t})\right)\Big \vert \leq (d-1)^{n} q^{n/2}.$$
\end{lemma}

\subsection{The character sum $\mathfrak{C} \left(m, {\bf x}\right)$}
In this subsection,  we will analyse the following character sum 
\begin{equation} \label{maincharsum}
\mathfrak{C} \left(m, {\bf x}\right):  = \mathop{\sum \sum \sum}_{\beta_1, \beta_2, \beta_3 \bmod m }   \chi_{m} \big(  F(\beta_1, \beta_2, \beta_3)\big) e \left(  \frac{x_1 \beta_1 +x_2 \beta_2 + x_3  \beta_3 }{m}\right),
\end{equation}
where $ m \in \mathbb{N}$ and ${\bf x}=(x_{1},x_{2},x_{3}) \in \mathbb{Z}^{3}$. The above character sum satisfies  multiplicative property. More precisely, we have the following lemma. 
\begin{lemma}  \label{multiprop}
Let $\mathfrak{C} \left(m, {\bf x}\right)$ be as in \eqref{maincharsum}. Then $\mathfrak{C} \left(m, {\bf x}\right)$ is a multiplicative function, i.e., we have
$$\mathfrak{C}(mn, {\bf x}) = \mathfrak{C}(m, {\bf x}) \, \mathfrak{C}(n, {\bf x}),$$
whenever $(m, n)=1$.
\end{lemma}
\begin{proof}
	We have
		\begin{align*}
		\mathfrak{C} \left(mn, {\bf x}\right)  = \mathop{\sum \sum \sum}_{\beta_1, \beta_2, \beta_3 \bmod mn }   \chi_{mn} \big(  F(\beta_1, \beta_2, \beta_3)\big) e \left(  \frac{x_1 \beta_1 +x_2 \beta_2 + x_3  \beta_3 }{mn}\right).
	\end{align*}
	 We write $\beta_i$ as
	 $$\beta_i=\beta_i^{\prime}n+\beta_i^{\prime \prime}m,  \ \ \quad \beta_i^\prime \bmod m \ \ \text{and} \ \ \beta_i^{\prime\prime} \bmod n, \ \ i=1,2,3. $$ 
	 
We observe that
	 \begin{align*}
	 	\chi_{mn}(...)&=\chi_{m} \big(  F(\beta_1^\prime n, \beta_2^\prime n, \beta_3^\prime n) \big) \ \chi_{n} \big(  F(\beta_1^{\prime\prime} m, \beta_2^{\prime\prime} m, \beta_3^{\prime\prime} m)\big) \\
	 	&=\chi_{m} \big(  F(\beta_1^\prime, \beta_2^\prime , \beta_3^\prime ) \big) \ \chi_{n} \big(  F(\beta_1^{\prime\prime} , \beta_2^{\prime\prime} , \beta_3^{\prime\prime} )\big) 
	 \end{align*}
	 and 
	 $$e \left(...\right)=e \left(  \frac{x_1 \beta_1^{\prime} +x_2 \beta_2^{\prime} + x_3  \beta_3^{\prime} }{m}\right) \ e \left(  \frac{x_1 \beta_1^{\prime \prime} +x_2 \beta_2^{\prime \prime} + x_3  \beta_3^{\prime \prime} }{n}\right).$$
	 Hence the lemma follows.
	 	\end{proof}
 	Thus it is sufficient to study $\mathfrak{C}(m, {\bf x})$ whenever $m$ is a prime power. Infact, as we are ultimately interested in $\mathfrak{C}(qq^{\prime}, {\bf x})$, it is enough to study $\mathfrak{C}(p, {\bf x})$ for $p$ prime.  The following lemma provides required bounds for such sums.  

\begin{lemma} \label{c2bound}
Let $p$ a prime and ${\bf x} = (x_{1}, x_{2}, x_{3})$ be any triplet of integers. Then we have square root cancellations in $\mathfrak{C}(p, {\bf x})$, i.e.,

\begin{align*}
	\mathfrak{C}(p, {\bf x}) \ll p^{3/2}.
\end{align*}
\end{lemma}

\begin{proof}
Consider the character sum 

\begin{align*}
	\mathfrak{C} \left(p, {\bf x}\right)  &= \mathop{\sum \sum \sum}_{\beta_1, \beta_2, \beta_3 \bmod p}   \chi_{p} \big(  F(\beta_1, \beta_2, \beta_3)\big) e \left(  \frac{x_1 \beta_1 +x_2 \beta_2 + x_3  \beta_3 }{p}\right).
	\end{align*}
First we consider the case when $p|x_i$ for all $i$. In this case the above sum looks like 
\begin{align*}
\mathfrak{C} \left(p, {\bf x}\right)  &= \mathop{\sum \sum \sum}_{\beta_1, \beta_2, \beta_3 \bmod p}   \chi_{p} \big(  F(\beta_1, \beta_2, \beta_3)\big).
\end{align*}
Now applying Lemma \ref{multibound}, we see that the above sum is bounded by $p^{3/2}$.
Now we consider the other case, i.e., $x_i\not \equiv 0  \bmod p$ for some $i$. Without loss of generality, let's assume  $x_2\not \equiv 0  \bmod p$. In this case, we split $\mathfrak{C} \left(p, {\bf x}\right)$  as follows 
\begin{align}\label{C1+C2}
		\mathfrak{C} \left(p, {\bf x}\right) 
		=&\mathop{ \sum \sum}_{ \beta_2, \beta_3 \bmod p} ...+\mathop{\sum \sum \sum}_{\substack{\beta_1, \beta_2, \beta_3 \bmod p \\ \beta_1 \not \equiv 0 \bmod  p}} ...
		\end{align}
We now  consider the first term of the above expression which is given by
\begin{align*}
	&\mathop{{\sum \sum}}_{\beta_2, \beta_3 \bmod p}\chi_{p} \big(  F(0, \beta_2, \beta_3)\big) e \left(  \frac{x_2 \beta_2 + x_3  \beta_3 }{p}\right) \\
=& \mathop{\sideset{}{^\star}{\sum } {\sum}}_{\beta_2, \beta_3 \bmod p}\chi_{p} \big(  F(0, 1, \beta_3)\big) e \left(  \frac{(x_2  + x_3  \beta_3)\beta_2 }{p}\right) \\
=&-\mathop{{\sum}}_{\beta_3 \bmod p} \chi_{p} \big(  F(0, 1, \beta_3)\big)+p\mathop{{\sum}}_{\substack{\beta_3 \bmod p \\ x_2+x_3\beta_3 \equiv 0 \bmod p}} \chi_{p} \big(  F(0, 1, \beta_3)\big).
\end{align*}
Note that when $p \mid \beta_2$, then $\chi_{p} \big(  F(0, \beta_2, \beta_3)\big)=0$ by the nature of the form $F$. The congruence condition determine $\beta_3$ uniquely. Hence the above expression is bounded by $p$. Now we consider the second sum of \eqref{C1+C2}
\begin{align*}
	&\mathop{\sum \sum \sum}_{\substack{\beta_1, \beta_2, \beta_3 \bmod p \\ \beta_1 \not \equiv 0 \bmod  p}}\chi_{p} \big(  F(\beta_1, \beta_2, \beta_3)\big) e \left(  \frac{x_1 \beta_1 +x_2 \beta_2 + x_3  \beta_3 }{p}\right) \\
	=&\mathop{\sum \sum \sum}_{\substack{\beta_1, \beta_2, \beta_3 \bmod p \\ \beta_1 \not \equiv 0 \bmod  p}}\chi_{p} \big(  F(1, \beta_2, \beta_3)\big) e \left(  \frac{(x_1 +x_2 \beta_2 + x_3  \beta_3)\beta_1 }{p}\right) \\
	=-&\mathop{ \sum \sum}_{\substack{\beta_2, \beta_3 \bmod p}}\chi_{p} \big(  F(1, \beta_2, \beta_3)\big)+p\mathop{\sum \sum}_{\substack{ \beta_2, \beta_3 \bmod p \\ x_1+x_2\beta_2+x_3\beta_3 \equiv 0 \bmod  p}}\chi_{p} \big(  F(1, \beta_2, \beta_3)\big). 
\end{align*}
We apply Lemma \ref{multibound} to the first term. Hence, it is bounded by $p$. In the second term, given $\beta_3$, the congruence condition determine $\beta_2$ uniquely in terms of $\beta_3$. Finally apply Lemma \ref{multibound} to sum over $\beta_3$, we see that the above expression is bound by $p^{3/2}$. Hence we have the lemma.
\end{proof}

\subsection{The character sum $\mathscr{C}(p, {\bf x})$}
In this subsection we will analyze  the character sum
\begin{equation} \label{finalchar}
\mathscr{C}(p, {\bf x}) := \mathop{\sum \sum \sum }_{\alpha_{1}, \alpha_{2}, \alpha_{3} \, \text{mod} \, p} \, \mathfrak{C}(p, {\bf \alpha}) \, e\left(\frac{\alpha_{1}x_{1}+\alpha_{2}x_{2}+\alpha_{3}x_{3}}{p}\right).
\end{equation}
where $\mathfrak{C}(p, {\bf \alpha})$ is as given in \eqref{maincharsum} with ${\bf \alpha}=(\alpha_{1},\alpha_{2},\alpha_{3})$ and ${\bf x}=(x_{1},x_{2},x_{3}) \in \mathbb{Z}^{3}$. We will encounter the above character sum later.
\begin{lemma} \label{finalcharsum}
We have
\begin{align*}
	\mathscr{C}(p, {\bf x}) \ll p^3.
\end{align*}
\end{lemma}

\begin{proof}
 Plugging  \eqref{maincharsum} into \eqref{finalchar}, we see that the character sum $\mathscr{C}(p, {\bf x})$ is given by
\begin{align*}
	\mathop{\sum \sum \sum }_{\alpha_{1}, \alpha_{2}, \alpha_{3} \, \text{mod} \, p}& \mathop{\sum \sum \sum}_{\beta_1, \beta_2, \beta_3 \bmod p}   \chi_{p} \big(  F(\beta_1, \beta_2, \beta_3)\big) \\
	&\times e\left(\frac{\alpha_{1}(x_{1}-\beta_1)+\alpha_{2}(x_2-\beta_2)+\alpha_{3}(x_3-\beta_3)}{p}\right).
\end{align*}
Now executing the sum over $\alpha_{i}$'s we get congruence conditions $\beta_i \equiv x_i  \bmod p$, for $i=1,2$ and $3$, which determine $\beta_i's $ uniquely. Hence the above sum is bounded by $p^3$.
\end{proof}

\section{\bf Estimates for $\mathcal{S}(Q, B)$}
We now decompose $\mathcal{S} (Q,B)$ in \eqref{S after poisson} as follows:
$$\mathcal{S} (Q,B) = \mathcal{S} ^{\flat}(Q,B) + \mathcal{S} ^{\sharp}(Q, B),$$
where
$$\mathcal{S}^{\flat}(Q,B) = \mathop{\sum \sum }_{\substack{q,q^{\prime} \\ q \neq q^{\prime} \\ (q,q^{\prime})>1}} (...), \quad \, \text{and} \, \, \quad \mathcal{S}^{\sharp}(Q, B) = \mathop{\sum \sum }_{\substack{q,q^{\prime}  \\ (q,q^{\prime})=1}}(...).$$

\subsection{Estimation of $\mathcal{S}^{\sharp}(Q, B)$} We recall that
$$\mathcal{S}^{\sharp}(Q, B) = \frac{B^{3}}{Q^{2}} \mathop{\sum \sum}_{\substack{q, q^{\prime} \sim^{\star} Q \\ (q,q^{\prime})=1 }} \frac{1}{(qq^{\prime})^3} \mathop{\sum \sum \sum}_{x_{1}, x_{2},x_{3} \ll \frac{Q^2}{B}B^{\epsilon}} \mathfrak{C}(qq^{\prime}, {\bf x}) \, \mathfrak{I}(qq^{\prime}, {\bf x}).$$

By the multiplicative property of $\mathfrak{C}(...)$, as given in Lemma \ref{multiprop}, it follows that
\begin{align*}
\mathcal{S}^{\sharp}(Q, B)& = \frac{B^{3}}{Q^{2}} \mathop{\sum \sum \sum}_{x_{1}, x_{2},x_{3} \ll \frac{Q^2}{B}B^{\epsilon}  }  \mathop{\sum \sum}_{p_{1} \sim P_{1}, p_{2} \sim P_{2}} \frac{1}{(p_{1}p_{2})^3} \mathfrak{C}\left(p_{1}, {\bf x}\right) \mathfrak{C}\left(p_{2}, {\bf x}\right)\\
& \quad \quad \quad \times \mathop{\sum \sum}_{\substack{p_{1}^{\prime}  \sim P_{1}, p_{2}^{\prime} \sim P_{2} \\  p_{1}^{\prime} \neq p_{1} \\ p_{2}^{\prime} \neq p_{2}}} \frac{1}{(p_{1}^{\prime} p_{2}^{\prime})^{3}}  \mathfrak{C}\left(p_{1}^{\prime}, {\bf x}\right) \mathfrak{C}\left(p_{2}^{\prime}, {\bf x}\right)\, \mathfrak{I}(p_{1}p_{1}^{\prime} p_{2} p_{2}^{\prime} ,{\bf x}). 
\end{align*}

We further decompose the above sum as
\begin{equation} \label{split1}
\mathcal{S}^{\sharp}(Q, B) = \mathcal{S}_{1}^{\sharp }(Q, B) - \mathcal{S}_{2}^{\sharp }(Q, B)- \mathcal{S}_{3}^{\sharp }(Q, B) + 2 \mathcal{S}_{4}^{\sharp }(Q, B),
\end{equation}
where 
$$\mathcal{S}_{1}^{\sharp }(Q, B) = \frac{B^{3}}{Q^{2}} \mathop{\sum \sum \sum}_{x_{1}, x_{2},x_{3} \ll \frac{Q^2}{B}B^{\epsilon}  }  \mathop{\sum \sum}_{p_{1} \sim P_{1}, p_{2} \sim P_{2}} \mathop{\sum \sum}_{\substack{p_{1}^{\prime} \sim P_{1}, p_{2}^{\prime} \sim P_{2}}} (...),$$

$$\mathcal{S}_{2}^{\sharp }(Q, B) = \frac{B^{3}}{Q^{2}} \mathop{\sum \sum \sum}_{x_{1}, x_{2},x_{3} \ll \frac{Q^2}{B}B^{\epsilon}  }  \mathop{\sum \sum}_{p_{1} \sim P_{1}, p_{2} \sim P_{2}}  \sum_{p_{1}^{\prime} \sim P_{1}} \sum_{\substack{p_{2}^{\prime} \sim P_{2} \\ p_{2}^{\prime} = p_{2}}} (...),$$

$$\mathcal{S}_{3}^{\sharp }(Q, B) = \frac{B^{3}}{Q^{2}} \mathop{\sum \sum \sum}_{x_{1}, x_{2},x_{3} \ll \frac{Q^2}{B}B^{\epsilon}  }  \mathop{\sum \sum}_{p_{1} \sim P_{1}, p_{2} \sim P_{2}}   \sum_{\substack{p_{1}^{\prime} \sim P_{1} \\ p_{1}^{\prime} = p_{1}}} \sum_{p_{2}^{\prime} \sim P_{2}} (...),$$
and 
$$\mathcal{S}_{4}^{\sharp }(Q, B) = \frac{B^{3}}{Q^{2}} \mathop{\sum \sum \sum}_{x_{1}, x_{2},x_{3} \ll \frac{Q^2}{B}B^{\epsilon}  }  \mathop{\sum \sum}_{p_{1} \sim P_{1}, p_{2} \sim P_{2}}   \sum_{\substack{p_{1}^{\prime} \sim P_{1} \\ p_{1}^{\prime} = p_{1}}} \sum_{\substack{p_{2}^{\prime} \sim P_{2} \\ p_{2}^{\prime} = p_{2}}}(...).$$

We now analyze $\mathcal{S}_{3}^{\sharp }(Q, B)$ and $\mathcal{S}_{4}^{\sharp }(Q, B)$ in the following lemma .

\begin{lemma} \label{s34sharp}
Let $\mathcal{S}_{3}^{\sharp }(Q, B)$ and $\mathcal{S}_{4}^{\sharp }(Q, B)$  be as in \eqref{split1}. Then we have
$$\mathcal{S}_{3}^{\sharp }(Q, B) \ll P_{1}^2 P_{2}^{3},$$
and
$$\mathcal{S}_{4}^{\sharp }(Q, B) \ll P_{1}^2 P_{2}^{2}.$$
\end{lemma}
\begin{proof}
Consider the sum $\mathcal{S}_{3}^{\sharp }(Q, B)$, which is given by

\begin{align*}
& \frac{B^{3}}{Q^{2}} \mathop{\sum \sum \sum}_{x_{1}, x_{2},x_{3} \ll \frac{Q^2}{B}B^{\epsilon}  }  \mathop{\sum \sum}_{p_{1} \sim P_{1}, p_{2} \sim P_{2}} \frac{1}{(p_{1}p_{2})^3} \mathfrak{C}\left(p_{1}, {\bf x}\right) \mathfrak{C}\left(p_{2}, {\bf x}\right)\\
& \quad \quad \quad \times \mathop{\sum \sum}_{\substack{p_{1}^{\prime}  \sim P_{1}, p_{2}^{\prime} \sim P_{2}  \\ p_{1}^{\prime} = p_{1}}} \frac{1}{(p_{1}^{\prime} p_{2}^{\prime})^{3}}  \mathfrak{C}\left(p_{1}^{\prime}, {\bf x}\right) \mathfrak{C}\left(p_{2}^{\prime}, {\bf x}\right)\, \mathfrak{I}(p_{1}p_{1}^{\prime} p_{2} p_{2}^{\prime} ,{\bf x}).
\end{align*} 
Now using Lemma \ref{c2bound} to bound the character sums and executing all the sum trivially we get the required bound for $\mathcal{S}_{3}^{\sharp }(Q, B)$. Analysing $\mathcal{S}_{4}^{\sharp }(Q, B)$ in a similar way, we get bounds for this sum. Hence the lemma follows.
\end{proof}

In the following lemma we will estimate $\mathcal{S}_{1}^{\sharp }(Q, B)$ and $\mathcal{S}_{2}^{\sharp }(Q, B)$. In fact we  get better bounds for  $\mathcal{S}_{2}^{\sharp }(Q, B)$ compared to $\mathcal{S}_{1}^{\sharp }(Q, B)$. 

\begin{lemma} \label{midlelemma}
Let $\mathcal{S}_{1}^{\sharp }(Q, B)$ and $\mathcal{S}_{2}^{\sharp }(Q, B)$ be as in \eqref{split1}. Then we have
$$\mathcal{S}_{1}^{\sharp}(Q, B) \ll    \mathcal{S}_{1}^{\sharp \sharp}(Q, B)$$
and
$$\mathcal{S}_{2}^{\sharp}(Q, B) \ll \frac{\mathcal{S}_{1}^{\sharp \sharp}(Q, B)}{P_{2}},$$
where $\mathcal{S}_{1}^{\sharp \sharp}(Q, B)$ is given by
\begin{align} \label{sdoublesharp}
 \frac{B^{3 +\epsilon}}{Q^2 P_{2}P_{1}^6} \mathop{\sum \sum}_{p_{1}, p_{1}^{\prime}  \sim P_{1}}   \Big \vert \mathop{\sum \sum \sum}_{\substack{x_{1}, x_{2},x_{3} \in \mathbb{Z} } } \mathfrak{C}(p_{1},{\bf x}) \,  \mathfrak{C}(p_{1}^{\prime},{\bf x})   V \left( \frac{Bx_{1}}{Q^2 B^{\epsilon}}, \frac{Bx_{2}}{Q^2 B^{\epsilon}}, \frac{Bx_{3}}{Q^2 B^{\epsilon}}\right)  \Big \vert, 
\end{align}
where $V$ is a smooth bump function.   
\end{lemma}

\begin{proof}
By making a change of variable $y_{i} \rightarrow p_{1} p_{1}^{\prime} y_{i}$ for $i=1,2,3$ in \eqref{theintegral} we see that $\mathfrak{I}(p_{1}p_{1}^{\prime} p_{2}p_{2}^{\prime},{\bf x})$ tranforms into
$$ (p_{1}p_{1}^{\prime})^3 \iiint_{\mathbb{R}^3} W\left(p_{1}p_{1}^{\prime} y_{1}, p_{1}p_{1}^{\prime}y_{2}, p_{1}p_{1}^{\prime} y_{3}\right) \, e \left(  \frac{-B(x_1 y_1 + x_2 y_2 +x_3 y_3)}{p_{2}p_{2}^{\prime}}\right) d y_1 dy_2 dy_3 .$$

Upon substituting  this expression in place of  $\mathfrak{I}(p_{1}p_{1}^{\prime} p_{2}p_{2}^{\prime},{\bf x})$ in $\mathcal{S}_{1}^{\sharp}(Q, B)$, we see that $\mathcal{S}_{1}^{\sharp}(Q, B)$ is bounded by  
\begin{align*}
 & \frac{B^{3}}{Q^{2}} \iiint_{[\frac{-2}{P_{1}^2},\frac{2}{P_{1}^2}]^{3}}  \mathop{\sum \sum \sum}_{x_{1}, x_{2},x_{3} \ll \frac{Q^2}{B}B^{\epsilon}  }  \Big \vert \mathop{\sum \sum}_{p_{1} \sim P_{1}, p_{1}^{\prime} \sim P_{1}}  \mathfrak{C}\left(p_{1}, {\bf x}\right) \, \mathfrak{C}\left(p_{1}^{\prime}, {\bf x}\right)\, W\left(p_{1}p_{1}^{\prime} y_{1},.., p_{1}p_{1}^{\prime} y_{3}\right) \Big \vert\\
& \times \Big \vert \mathop{\sum \sum}_{p_{2}, p_{2}^{\prime} \sim P_{2}} \frac{1}{(p_{2}p_{2}^{\prime})^{3}}  \,\mathfrak{C}\left( p_{2}, {\bf x}\right) \,\mathfrak{C}\left( p_{2}^{\prime}, {\bf x}\right) \, e \left( -  \frac{x_1 y_1  B + x_2 y_2 B  +  x_3 y_3  B }{p_{2} p_{2}^\prime}\right)  \Big \vert d y_1 \ dy_2 \ dy_3 .
\end{align*}

We estimate the terms written after $\times$  in the above expression using  $$\mathfrak{C}(p_{2},{\bf x}), \mathfrak{C}(p_{2}^{\prime},{\bf x}) \ll P_{2}^{3/2} $$  from Lemma \ref{c2bound}. Therefore we get

\begin{align}\label{S1 in lemma 52}
\mathcal{S}_{1}^{\sharp}(Q, B) \leq &  \frac{B^{3}}{Q^{2} P_{2}} \iiint_{[\frac{-2}{P_{1}^2},\frac{2}{P_{1}^2}]^{3}}  \mathop{\sum \sum \sum}_{x_{1}, x_{2},x_{3} \ll \frac{Q^2}{B}B^{\epsilon}  } \notag \\
&\times  \Big \vert \mathop{\sum \sum}_{p_{1} \sim P_{1}, p_{1}^{\prime} \sim P_{1}}  \mathfrak{C}\left(p_{1}, {\bf x}\right) \, \mathfrak{C}\left(p_{1}^{\prime}, {\bf x}\right)\, W\left(p_{1}p_{1}^{\prime} y_{1}, p_{1}p_{1}^{\prime}y_{2}, p_{1}p_{1}^{\prime} y_{3}\right) \Big \vert dy_{1} \, dy_{2} \, dy_{3}.
\end{align}
We use the Mellin inversion formula to separate  $p_{i}$ from $p_{i}^{\prime}$ in  the weight function $W(p_{1}p_{1}^{\prime} y_{1}, p_{1}p_{1}^{\prime}y_{2}, p_{1}p_{1}^{\prime} y_{3})$. Infact, we may also assume that 
$$W(y_{1}, y_{2}, y_{3}) = W_{1}(y_{1}) W_{2} (y_{2}) W_{3}(y_{3}),$$ and hence $W(p_{1}p_{1}^{\prime} y_{1}, p_{1}p_{1}^{\prime}y_{2}, p_{1}p_{1}^{\prime} y_{3})$ can be rewritten  as


\begin{align}\label{w after melin}
	&\frac{1}{(2 \pi  )^3} \iiint_{\mathbb{R}^3} \widehat{W}_{1}( it_{1}) \widehat{W}_{2}( it_{2}) \widehat{W}_{3}( it_{3}) \notag \\
	& \times  \left(p_{1}p_{1}^{\prime} \right)^{- i(t_{1}+t_{2}+t_{3})}y_{1}^{- it_{1}} y_{2}^{-i t_{2}} y_{3}^{ -it_{3}} dt_{1}dt_{2} dt_{3},
\end{align}
where $\widehat{W}_{j}(s)$ is the Mellin transform of $W_{j}$ for $j =1,2,3$. Since $W_{j}$'s are smooth bump functions, it can be seen easily (using integration by parts repeatedly) that $\widehat{W}_{j}( it_{j})$'s are negligibly small if $t_j \gg B^{\epsilon}$ for some $j$.  Therefore, after plugging \eqref{w after melin} in \eqref{S1 in lemma 52},  the expression inside $| \ |$ in \eqref{S1 in lemma 52} is dominated by 
\begin{align*}
	\iiint_{[-B^{\epsilon},  B^{\epsilon}]^{3}} \Big \vert  \sum_{p_{1} \sim P_{1}} \, \mathfrak{C}(p_{1}, {\bf x}) \, p_{1}^{ -i (t_{1} + t_{2}+t_{3})}\Big \vert^2 dt_{1}dt_{2} dt_{3}.
\end{align*}
  Hence, after estimating integral over $y_i$'s trivially, we see that $\mathcal{S}_{1}^{\sharp}(Q, B)$ is bounded by
$$  \frac{B^{3 +\epsilon}}{Q^2 P_{2}P_{1}^6} \iiint_{[-B^{\epsilon},  B^{\epsilon}]^{3}}\mathop{\sum \sum \sum}_{x_{1}, x_{2},x_{3} \ll \frac{Q^2}{B}B^{\epsilon} } \Big \vert  \sum_{p_{1} \sim P_{1}} \, \mathfrak{C}(p_{1}, {\bf x}) \, p_{1}^{- i (t_{1} + t_{2}+t_{3})}\Big \vert^2dt_{1}dt_{2} dt_{3}.$$
 Now we smooth out  the sum over $x_{i}$ by introducing a bump function $V$. Then we open the absolute value square, interchanging  $p_{1}, p_{1}^{\prime}$ sums with the sums over $x_{i}$'s, and  bounding integrals trivially, we get the first part of the lemma.
Similar analysis gives the second part of the lemma. In fact the extra saving $P_{2}$ comes from the fact that the sum over  $p_{2}^{\prime}$ is absent in $\mathcal{S}_{2}^{\sharp}(Q, B)$.  
\end{proof}

\begin{lemma} \label{afterpoisson}
We have
$$\mathcal{S}_{1}^{\sharp \sharp}(Q, B)  \ll P_{1}^2 P_{2}^{3} + \frac{P_{2}^3 B^{\epsilon}}{P_{1}^2} \mathop{\sum \sum }_{\substack{p_{1}, p_{1}^{\prime} \sim P_{1} \\ p_{1} \neq p_{1}^{\prime}}} \frac{1}{(p_{1}p_{1}^{\prime})^3} \, \mathop{\sum \sum \sum}_{x_{1}, x_{2}, x_{3} \ll \frac{p_{1}p_{1}^{\prime} B}{Q^2} B^{\epsilon}} \, \Big \vert \mathcal{C}(x_{1},x_{2},x_{3}; p_{1}, p_{2}) \Big \vert,$$
where $\mathcal{C}(x_{1},x_{2},x_{3}; p_{1}, p_{2})$ is as given in \eqref{charsum}.

\end{lemma}

\begin{proof}
Recall from \eqref{sdoublesharp} that the expression for $\mathcal{S}_{1}^{\sharp \sharp}(Q, B)$ is given by
\begin{align*}
	\frac{B^{3 +\epsilon}}{Q^2 P_{2}P_{1}^6} \mathop{\sum \sum}_{p_{1}, p_{1}^{\prime}  \sim P_{1}}   \Big \vert \mathop{\sum \sum \sum}_{\substack{x_{1}, x_{2},x_{3} \in \mathbb{Z} } } \mathfrak{C}(p_{1},{\bf x}) \,  \mathfrak{C}(p_{1}^{\prime},{\bf x})   V \left( \frac{Bx_{1}}{Q^2 B^{\epsilon}}, \frac{Bx_{2}}{Q^2 B^{\epsilon}}, \frac{Bx_{3}}{Q^2 B^{\epsilon}}\right)  \Big \vert, 
\end{align*}
where $V(x,y,z)$ is a bump function. By Lemma \ref{c2bound}, it follows that the diagonal $p_{1} = p_{1}^{\prime}$ contributes at most
$$P_{1}^2 P_{2}^3$$
to  $\mathcal{S}_{1}^{\sharp \sharp}(Q, B)$, thus getting the first term in the statement of the lemma.  In the off-diagonal case, that is when $p_{1} \neq p_{1}^{\prime}$, we apply the Poisson summation formula to the sum over $x_{i}$'s . After an application of the Poisson summation formula,  the sum over $x_{i}$ in $\mathcal{S}_{1}^{\sharp \sharp}(Q, B)$ becomes

$$\frac{Q^6 B^{\epsilon}}{B^3 (p_{1}p_{1}^{\prime})^3} \mathop{\sum \sum \sum }_{x_{1}, x_{2}, x_{3} \in \mathbb{Z}} \mathcal{C}(x_{1},x_{2},x_{3}, \gamma_{3} ,\gamma_{3}^{\prime}) \, \mathscr{I}(x_{1},x_{2},x_{3})$$
where

\begin{align}
\mathcal{C}(...) = \mathop{\sum \sum \sum }_{\substack{\alpha_{1} \, \alpha_{2}, \alpha_{3} \, \text{mod} \, p_{1} p_{1}^{\prime} }} \mathfrak{C}(p_{1},{\bf \alpha}) \,  \mathfrak{C}(p_{1}^{\prime},{\bf \alpha}) \, e\left(\frac{\alpha_{1} x_{1}+ \alpha_{2} x_{2}+\alpha_{3}x_{3}}{p_{1}p_{1}^{\prime}}\right) \label{charsum}
\end{align}

and
$$\mathscr{I}(...) = \iiint_{\mathbb{R}^3} \, V(y_{1}, y_{2}, y_{3}) \, e\left(-\frac{(x_{1}y_{1}+x_{2}y_{2}+x_{3}y_{3}) (Q^2B^{\epsilon}/B)}{p_{1}p_{1}^{\prime}}\right) \, dy_{1} \, dy_{2} \, dy_{3}.$$
Upon using  integration by parts repeatedly we see that $\mathscr{I}(...)$ is negligibly small unless
$$x_{i} \ll \frac{p_{1}p_{1}^{\prime} B}{Q^2} B^{\epsilon}, \quad \,  \text{for all} \, i=1,2,3.$$
Note that we have $\vert \mathscr{I}(...) \vert \leq 1$. Hence we obtain the lemma.

\end{proof}

\begin{lemma}
Let ${\bf x}=(x_{1},x_{2},x_{3}) \in \mathbb{Z}^{3}$ be any triplet of integers. For any distinct primes $p_{1}, p_{1}^{\prime}$, we have
$$\mathcal{C}(x_{1},x_{2},x_{3}; p_{1}, p_{1}^{\prime}) = \mathscr{C}(p_{1},\bar{p_{1}^{\prime}} {\bf x}) \, \,  \mathscr{C}(p_{1}^{\prime}, \bar{p_{1}} {\bf x})$$
where $\ell {\bf x}:= (\ell x_{1}, \ell x_{2}, \ell x_{3})$ and $\mathscr{C}(p, {\bf x})$ is as given in \eqref{finalcharsum}.

\end{lemma}

\begin{proof}

The lemma follows by noting the fact that  any $\alpha_{i} \, \rm mod \, (p_{1} p_{1}^{\prime})$  in \eqref{charsum} can be written as
$$\alpha_{i} = \alpha_{i}^{\prime \prime} \, p_{1}^{\prime}  + \alpha_{i}^{\prime} \, p_{1}  , \quad \, \alpha_{i}^{\prime} \, \rm mod \, p_{1}^{\prime}, \,\, \alpha_{i}^{\prime \prime} \, \rm mod \, p_{1}$$
where $\bar{p}_{1}$ is the inverse of $p_{1}$ modulo $p_{1}^{\prime}$ and $\bar{p}_{1}^{\prime}$ is the inverse of $p_{1}^{\prime}$ modulo $p_{1}$.
\end{proof}


\begin{lemma} \label{s1sharpbound}
We have
\begin{equation}
\mathcal{S}_{1}^{\sharp}(Q, B)  \ll P_{1}^2 P_{2}^{3} + \frac{B^{3}}{P_{2}^{3}} B^{\epsilon}.
\end{equation}
\end{lemma}

\begin{proof}
From Lemma \ref{midlelemma}, Lemma \ref{afterpoisson} and Lemma \ref{multiprop}, we infer that
$$\mathcal{S}_{1}^{\sharp}(Q, B) \ll P_{1}^2 P_{2}^{3} + \frac{P_{2}^3 B^{\epsilon}}{P_{1}^2} \mathop{\sum \sum }_{\substack{p_{1}, p_{1}^{\prime} \sim P_{1} \\ p_{1} \neq p_{1}^{\prime}}} \frac{1}{(p_{1}p_{1}^{\prime})^3} \, \mathop{\sum \sum \sum}_{x_{1}, x_{2}, x_{3} \ll \frac{p_{1}p_{1}^{\prime} B}{Q^2} B^{\epsilon}} \, \Big \vert \mathscr{C}(p_{1},\bar{p_{1}^{\prime}} {\bf x}) \, \,  \mathscr{C}(p_{1}^{\prime}, \bar{p_{1}} {\bf x}) \Big \vert.$$
 By Lemma \ref{finalcharsum}, we have
\begin{align*}
	\mathscr{C}(p_{1},\bar{p_{1}^{\prime}} {\bf x}) \ll p_1^3.
\end{align*}
Similar bounds holds for $\mathscr{C}(p_{1}^{\prime}, \bar{p_{1}} {\bf x})$ as well. 
Using these bounds we see that 
$$\mathop{\sum \sum \sum}_{x_{1}, x_{2}, x_{3} \ll \frac{p_{1}p_{1}^{\prime} B}{Q^2} B^{\epsilon}} \, \Big \vert \mathscr{C}(p_{1},\bar{p_{1}^{\prime}} {\bf x}) \, \,  \mathscr{C}(p_{1}^{\prime}, \bar{p_{1}} {\bf x}) \Big \vert \ll  \frac{P_{1}^{6}B^{3}}{P_{2}^{6}}.$$

Finally executing  the sum over $x_{i}$, we get the lemma.
\end{proof}


\subsection{Estimation of $\mathcal{S}^{\flat}(Q, B)$ } \label{sectionflat}
In this section we estimate $\mathcal{S}^{\flat}(Q, B)$. For this we get better bounds than $\mathcal{S}^{\sharp}(Q, B)$. Indeed , we have the following lemma.

\begin{lemma} \label{flatbound}
We have
\begin{equation} 
\mathcal{S}^{\flat}(Q,B) \ll  \frac{P_{1}^3}{P_{2}}.
\end{equation}
\end{lemma}

\begin{proof}
We have
$$\mathcal{S}^{\flat }(Q, B) = \mathcal{S}_{1}^{\flat}(Q,B) + \mathcal{S}_{2}^{\flat}(Q,B) $$
where
$$\mathcal{S}_{1}^{\flat}(Q,B)=\mathop{\sum \sum \sum}_{x_{i}} \mathop{\sum \sum}_{p_{1} \sim P_{1}, p_{2} \sim P_{2}}  \sum_{\substack{p_{1}^{\prime} \sim P_{1} \\ p_{1}^{\prime} \neq p_{1}}} \sum_{\substack{p_{2}^{\prime} \sim P_{2} \\ p_{2}^{\prime} = p_{2}}} (...)$$
and 
$$\mathcal{S}_{2}^{\flat}(Q,B) =\mathop{\sum \sum \sum}_{x_{i}} \mathop{\sum \sum}_{p_{1} \sim P_{1}, p_{2} \sim P_{2}}    \sum_{\substack{p_{1}^{\prime} \sim P_{1} \\ p_{1}^{\prime} = p_{1}}} \sum_{\substack{p_{2}^{\prime} \sim P_{2} \\ p_{2}^{\prime} \neq p_{2}}}(...).$$
	
We only have to deal with sum $\mathcal{S}_{1}^{\flat}(Q,B)$ as the treatment for the sum $\mathcal{S}_{2}^{\flat}(Q,B)$ is similar(in fact easier and getting better bounds). The expression for $\mathcal{S}_{1}^{\flat}(Q,B)$ is given by
$$\frac{B^{3}}{Q^2} \mathop{\sum  \sum}_{\substack{p_{1},p_{1}^{\prime} \sim P_{1} \\ p_{1} \neq p_{1}^{\prime}}} \sum_{p_{2} \sim P_{2}} \frac{1}{(p_{1}p_{1}^{\prime} p_{2}^2)^3} \mathop{\sum \sum \sum}_{\substack{x_{1},x_{2},x_{3} \in \mathbb{Z} \\ x_{i} \ll \frac{Q^2 B^{\epsilon}}{B}}} \, \mathfrak{C}(p_{1}p_{1}^{\prime}p_{2}^2, {\bf x}) \, \mathfrak{I}(p_{1}p_{1}^{\prime}p_{2}^2, {\bf x}).$$
Let us deal with $\mathfrak{C}(p_{1}p_{1}^{\prime}p_{2}^2, {\bf x})$. From \eqref{maincharsum}, we recall 
	
$$\mathfrak{C}(p_{1}p_{1}^{\prime}p_{2}^2, {\bf x}) =\mathop{\sum \sum \sum}_{\beta_1, \beta_2, \beta_3 \bmod p_{1}p_{1}^{\prime}p_{2}^2 }   \chi_{p_{1}p_{1}^{\prime}} \big(  F(\beta_1, \beta_2, \beta_3)\big) e \left(  \frac{x_1 \beta_1 +x_2 \beta_2 + x_3  \beta_3 }{p_{1}p_{1}^{\prime}p_{2}^{2}}\right).$$
By Lemma \ref{multiprop} we get
$$\mathfrak{C}(p_{1}p_{1}^{\prime}p_{2}^2, {\bf x}) = \mathfrak{C}(p_{1}, {\bf x}) \,  \mathfrak{C}(p_{1}^{\prime}, {\bf x})  \, \mathfrak{C}(p_{2}^2, {\bf x}).$$
We see that
	
\[
\mathfrak{C}(p_{2}^2, {\bf x}) = \prod_{i=1}^{3} \sum_{\beta  \, \text{mod}  \, p_{2}^2 } \, e\left(\frac{x_{i} \beta}{p_{2}^2}\right) =
\begin{cases}
p_{2}^6 &\text{if} \, p_{2}^{2} \mid x_{i} \, \text{for all} \, i=1,2,3\\
0 &\text{otherwise}. 
\end{cases}
\]
	
Therefore we get
$$\mathcal{S}_{1}^{\flat}(Q,B)= \frac{B^{3}}{Q^2} \mathop{\sum  \sum}_{\substack{p_{1},p_{1}^{\prime} \sim P_{1} \\ p_{1} \neq p_{1}^{\prime}}} \sum_{p_{2} \sim P_{2}} \frac{1}{(p_{1}p_{1}^{\prime} )^3} \mathop{\sum \sum \sum}_{\substack{x_{1},x_{2},x_{3} \in \mathbb{Z} \\ x_{i} \ll \frac{Q^2 B^{\epsilon}}{B} \\ p_{2}^{2} \mid x_{i} \forall i}} \, \mathfrak{C}(p_{1}, {\bf x} ) \mathfrak{C}(p_{1}^{\prime}, {\bf x})\, \mathfrak{I}(p_{1}p_{1}^{\prime}p_{2}^2, {\bf x}).$$

Now we use Lemma \ref{c2bound} to bound character sums $\mathfrak{C}(p_{1}^{\prime}, {\bf x}), \mathfrak{C}(p_{1}, {\bf x} )$, and by executing the remaining sums we infer that 
$$\mathcal{S}_{1}^{\flat}(Q,B) \ll  \frac{P_{1}^3}{P_{2}}.$$
Thus the lemma follows.
\end{proof}

\section{\bf Conclusion: Proof of Theorem \ref{mainth}} \label{choose}

We recall that 
\begin{align*}
\mathcal{S}(Q,B) &= \mathcal{S}^{\sharp}(Q,B)+\mathcal{S}^{\flat}(Q,B) \\
& =\mathcal{S}_{1}^{\sharp }(Q, B) - \mathcal{S}_{2}^{\sharp }(Q, B)- \mathcal{S}_{3}^{\sharp }(Q, B) + 2 \mathcal{S}_{4}^{\sharp }(Q, B) +\mathcal{S}^{\flat}(Q,B).
\end{align*}
From Lemma \ref{s34sharp}, Lemma \ref{midlelemma}, Lemma \ref{s1sharpbound} and Lemma \ref{flatbound} we conclude that
$$\mathcal{S}(Q,B) \ll P_{1}^2 P_{2}^{3} + \frac{B^{3}}{P_{2}^{3}} B^{\epsilon}   + \frac{P_{1}^3}{P_{2}}.$$

Therefore we have
\begin{equation} \label{finaleq}
N(B) \ll \frac{B^{3}}{P_{1} P_{2}} + P_{1}^2 P_{2}^{3} + \frac{B^{3}}{P_{2}^{3}}  + \frac{P_{1}^3}{P_{2}}.
\end{equation}

We choose $P_{1}$  and $ P_{2}$ such that 
$$\frac{B^{3}}{P_{1} P_{2}} > \frac{B^{3}}{P_{2}^{3}}$$
which is same as the condition  $ P_{1} < P_{2}^2$, thus we take $P_{1} = P_{2}^{2-\epsilon}$. For this choice of $P_{1}$ and $P_{2}$, the first two terms dominate the other terms in \eqref{finaleq}.

We get the optimal choice for $ P_{2}$ by equating
$$\frac{B^{3}}{P_{1} P_{2}} = P_{1}^2 P_{2}^{3} .$$
Therefore the optimal choice is $P_{2}= B^{3/10}$. Thus, we have
$$N(B) \ll B^{2+ \frac{1}{10}+\epsilon}.$$

Hence the proof of  Theorem \ref{mainth} follows.

\section*{\bf Acknowledgements}
Authors are grateful to Prof. Ritabrata Munshi for sharing his ideas with them and for his support and encouragement.  Authors are  thankful to Prof. Satadal Ganguly for his constant support and encouragements. Authors express their thanks to Prof. T. D. Browning, Prof. L. Pierce and Prof. Per Salberger  for showing their interest in authors work and for giving useful suggestions and comments. Finally,  first three authors would like to thank Stat-Math unit, Indian Statistical Institute Kolkata and the last author would like to thank IIT Kanpur for providing excellent research environment.

\end{document}